\documentclass[12pt,leqno,a4paper,final]{amsart}
\usepackage{amsmath, amsthm, verbatim, amsfonts, amssymb, color}
\usepackage{mathrsfs}
\usepackage{dsfont}
\usepackage[a4paper,margin=1in]{geometry}

\theoremstyle{plain}
\newtheorem{theorem}{Theorem}
\newtheorem{corollary}[theorem]{Corollary}
\newtheorem{lemma}[theorem]{Lemma}
\newtheorem{proposition}[theorem]{Proposition}

\newtheorem*{theorem*}{Theorem}
\newtheorem*{corollary*}{Corollary}
\newtheorem*{lemma*}{Lemma}
\newtheorem*{proposition*}{Proposition}

\theoremstyle{definition}
\newtheorem{remark}[theorem]{Remark}
\newtheorem*{remark*}{Remark}
\newtheorem{example}[theorem]{Example}
\newtheorem{definition}[theorem]{Definition}

\renewcommand\labelenumi{\textup{\alph{enumi})}}
\renewcommand\theenumi\labelenumi

\renewcommand{\Re}{\ensuremath{\operatorname{Re}}}
\renewcommand{\Im}{\ensuremath{\operatorname{Im}}}
\newcommand{\eup}{\mathrm{e}}
\newcommand{\iup}{\mathrm{i}}

\makeatletter
\def\@makefnmark{\hbox{(\@textsuperscript{\normalfont\@thefnmark})}}
\makeatother
\DeclareFontFamily{U}{mathx}{\hyphenchar\font45}
\DeclareFontShape{U}{mathx}{m}{n}{
      <5> <6> <7> <8> <9> <10>
      <10.95> <12> <14.4> <17.28> <20.74> <24.88>
      mathx10
      }{}
\DeclareSymbolFont{mathx}{U}{mathx}{m}{n}
\DeclareFontSubstitution{U}{mathx}{m}{n}
\DeclareMathAccent{\widecheck}{0}{mathx}{"71}
\DeclareMathAccent{\wideparen}{0}{mathx}{"75}

\newcommand\Ee{\mathds{E}}
\newcommand\nat{\mathds{N}}
\newcommand\rat{\mathds{Q}}
\newcommand\integer{\mathds{Z}}
\newcommand\real{{\mathds{R}}}
\newcommand\comp{{\mathds{C}}}
\newcommand\HH{{\mathds{H}}}
\newcommand\I{\mathds{1}}
\newcommand\Pp{\mathds{P}}
\newcommand\id{\mathop{\mathrm{id}}}
\newcommand\Per{\mathop{\mathrm{Per}}}
\newcommand\supp{\mathop{\mathrm{supp}}}
\newcommand\spann{\mathop{\mathrm{span}}}
\newcommand\rn{{\mathds{R}^n}}

\newcommand{\Dcal}{\mathcal{D}}
\newcommand{\Lcal}{\mathcal{L}}
\newcommand{\Rcal}{\mathcal{R}}
\newcommand{\Pcal}{\mathcal{P}}
\newcommand{\Scal}{\mathcal{S}}
\newcommand{\Tcal}{\mathcal{T}}
\newcommand{\boxperp}{[\perp]}

\newcommand{\scalp}[2]{\langle #1,\,#2\rangle}

\usepackage{xcolor,marginnote}

\begin{document}
\begin{flushright}\small
\underline{Appeared in} \emph{Mathematica Scandinavica}, \textbf{128(2)} (2022) 365--388.\\
\bigskip\bigskip
\end{flushright}

\title[Liouville properties for L\'evy generators]{\bfseries On the Liouville and strong Liouville properties for a class of non-local operators}

\author[D.~Berger]{David Berger}

\author[R.L.~Schilling]{Ren\'e L.\ Schilling}
\address{TU Dresden\\ Fakult\"{a}t Mathematik\\ Institut f\"{u}r Mathematische Stochastik\\ 01062 Dresden, Germany}
\email{david.berger2@tu-dresden.de}
\email{rene.schilling@tu-dresden.de}

\thanks{\emph{Acknowledgement}. We thank Moritz Kassmann for drawing our attention to the paper \cite{ali-et-al} of Alibaud, del Teso, Endal and Jakobsen. We are grateful to Espen Jakobsen who sent us the latest version of \cite{ali-et-al} and whose comments were most helpful. Bj\"{o}rn B\"{o}ttcher, Wojciech Cygan, Franziska K\"{u}hn, Mateusz Kwa\'snicki, Niels Jacob, Victoria Knopova and Zolt\'an Sasv\'ari read various earlier versions, pointed out mistakes and made valuable suggestions---a big thank you, too. The comments of an anonymous referee helped to improve the presentation of this paper. Financial support through the DFG-NCN Beethoven Classic 3 project SCHI419/11-1 is gratefully acknowledged.}

\begin{abstract}
    We prove a necessary and sufficient condition for the Liouville and strong Liouville properties of the infinitesimal generator of a L\'evy process and subordinate L\'evy processes. Combining our criterion with the necessary and sufficient condition obtained by Alibaud \emph{et al.}, we obtain a characterization of  (orthogonal subgroup of)  the set of zeros of the characteristic exponent of the L\'evy process.
\end{abstract}

\maketitle

A $C^2$-function $f:\rn\to\real$ is called \textbf{harmonic}, if $\Delta f = 0$ for the Laplace operator $\Delta$. The classical Liouville theorem states that any bounded harmonic function is constant. Often it is helpful to understand $\Delta f$ as a Schwartz distribution in $\Dcal'(\rn)$ and to re-formulate the Liouville problem in the following way: The operator $\Delta$ enjoys the \textbf{Liouville property} if
\begin{gather}\label{vor-e02}
    f\in L^\infty(\rn) \;\;\text{and}\;\; \forall \phi\in C_c^\infty(\rn)\::\:\scalp{\Delta f}{\phi} := \scalp{f}{\Delta \phi} = 0 \implies f\equiv\textup{const}
\end{gather}
holds; $\scalp{\cdot}{\cdot}$ denotes the (real) dual pairing used in the theory of distributions. An excellent account on the history and the importance of the Liouville property can be found in the paper \cite{ali-et-al} by Alibaud \emph{et al.} If the condition `$f\in L^\infty(\rn)$' in \eqref{vor-e02} can be replaced by `$f\geq 0$', we speak of the \textbf{strong Liouville property}.

Like \cite{ali-et-al} we are interested in the analogue of \eqref{vor-e02} for a class of non-local operators with constant `coefficients'. Recall that $\frac 12\Delta$ is the infinitesimal generator of Brownian motion. This is a diffusion process with stationary and independent increments and continuous paths. If we give up the continuity of the sample paths, and consider stochastic processes with independent, stationary increments and right-continuous paths with finite left-hand limits, we get the family of \textbf{L\'evy processes}, cf.\ \cite{sato,barca}. It is well known, see \cite{jac-1,barca}, that the infinitesimal generator $\Lcal_\psi$ of a L\'evy process (in any space $L^p(\rn)$, $1\leq p < \infty$ or in $C_\infty(\rn)= \overline{C_c^\infty(\rn)}^{\|\bullet\|_\infty}$, cf.\ \cite[Proposition 12.7]{ber-for} and \cite[Example 1.3.e), p.~4]{LM3}) is a \textbf{pseudo-differential operator}
\begin{gather}\label{vor-e04}
    \widehat{\Lcal_\psi u}(\xi) = -\psi(\xi)\widehat u(\xi),\quad u\in \Scal(\rn)
\end{gather}
where $\widehat u(\xi) = (2\pi)^{-n}\int_\rn \eup^{-\iup\xi\cdot x} u(x)\,dx$ is the Fourier transform and $\Scal(\rn)$ is the Schwartz space of rapidly decreasing smooth functions; the inverse Fourier transform is denoted by $\widecheck u$.  The \textbf{symbol} $\psi:\rn\to\comp$ is a \textbf{continuous and negative definite} function which is uniquely characterized by its \textbf{L\'evy--Khintchine representation}
\begin{gather}\label{vor-e06}
    \psi(\xi)
    = -\iup b\cdot\xi + \frac 12 Q\xi\cdot\xi + \int_{\rn\setminus\{0\}}\left(1-\eup^{\iup  \xi\cdot x}+\iup\xi\cdot x\I_{(0,1)}(|x|)\right)\nu(dx);
\end{gather}
the `coefficients' $b\in\rn$, $Q\in\real^{n\times n}$ (a positive semidefinite matrix) and $\nu$ (a Radon measure on $\rn\setminus\{0\}$ such that $\int_{\rn\setminus\{0\}} \min\{|x|^2,1\}\,\nu(dx)<\infty$) uniquely describe $\psi$. In probabilistic language, the symbol $\psi$ is known as the \textbf{characteristic exponent} of the L\'evy process $(X_t)_{t\geq 0}$ and $(b,Q,\nu)$ is the \textbf{L\'evy triplet}. This is due to the fact that the law $\mu_t(dy) := \Pp(X_t\in dy)$ of the random variable $X_t$ satisfies
\begin{gather}\label{vor-e08}
    \Ee \eup^{\iup \xi\cdot X_t} = \widecheck{\mu_t}(\xi) = \eup^{-t\psi(\xi)},\quad t>0,\;\xi\in\rn;
\end{gather}
the particular structure of the expected value is intimately linked with the property that the random variable $X_t$, resp., $\mu_t$ is infinitely divisible, cf.\ \cite{sato}. It is clear from \eqref{vor-e06} that $\xi\mapsto\psi(A\xi)$ is, for any invertible matrix $A\in\real^{n\times n}$, again continuous and negative definite; in particular, negative definiteness is preserved under coordinate changes. From \eqref{vor-e04} we can see that the adjoint $\Lcal_\psi^*$ is also a L\'evy generator whose symbol $\overline\psi$ is the complex conjugate of $\psi$: Let $u,\phi\in\Scal(\rn)$. By Plancherel's theorem
\begin{gather}\label{vor-e10}
    \scalp{\Lcal_\psi u}{\phi}
    =\scalp{\widehat{\Lcal_\psi u}}{\widecheck \phi}
    =\scalp{-\psi\widehat{u}}{\widecheck \phi}
    =\scalp{\widehat{u}}{-\psi\widecheck \phi}
    =\scalp{\widehat{u}}{\widecheck{\Lcal_{\overline\psi} \phi}}
    =\scalp{u}{\Lcal_{\overline\psi} \phi}.
\end{gather}
Therefore we have $\Lcal^*_\psi = \Lcal_{\overline\psi}$ (in any space $L^p(\rn)$, $p\in [1,\infty)$ and in $C_\infty(\rn):=\overline{C^\infty_c(\rn)}^{\|\cdot\|_\infty}$) and the corresponding L\'evy process is just $(-X_t)_{t\geq 0}$.

If we combine \eqref{vor-e04} and \eqref{vor-e06} we get a further representation for $\Lcal_\psi$ on $\Scal(\rn)$
\begin{gather}\label{vor-e12}\begin{aligned}
    \Lcal_\psi u(x)
    &= b\cdot\nabla u(x) + \frac 12\nabla \cdot Q\nabla u(x) \\
    &\qquad\mbox{} + \int_{\rn\setminus\{0\}} \left(u(x+y)-u(x)-y\cdot\nabla u(x)\I_{(0,1)}(|y|)\right)\nu(dy).
\end{aligned}\end{gather}
It is not hard to see that this class of operators includes the fractional Laplace operator $-(-\Delta)^\alpha$ for $\alpha\in (0,2)$.

The main results of our paper is the following version of a Liouville theorem for non-local operators of the type described above:
\begin{theorem*}[Liouville]
    Let $\Lcal_\psi$ be the generator of a L\'evy process with characteristic exponent $\psi$. The operator $\Lcal_\psi$ has the Liouville property, i.e.\
    \begin{gather}\label{vor-e14}
        f\in L^\infty(\rn)\;\;\text{and}\;\;\forall\phi\in C_c^\infty(\rn)\::\:\scalp{\Lcal_\psi f}{\phi} = 0
        \implies f\equiv\textup{const}
    \end{gather}
    if, and only if, the zero-set of the characteristic exponent satisfies $\{\eta\in\rn\mid\psi(\eta)=0\}=\{0\}$.

    If, in addition, the L\'evy measure $\nu$ satisfies $\int_{|y|\geq 1} g(y)\,\nu(dy)$ for some submultiplicative function $g:\rn\to [0,\infty)$ \textup{(}for the definition of a submultiplicative function see Definition \ref{gliou-23}\textup{)}, then
    \begin{gather}\label{vor-e15}
        0\leq f\leq g\;\;\text{and}\;\;\forall\phi\in C_c^\infty(\rn)\::\:\scalp{\Lcal_\psi f}{\phi} = 0
        \implies f\equiv\textup{const}
    \end{gather}
    if, and only if, $\{\eta\in\rn\mid\psi(\eta)=0\}\cup\{\eta\in\rn\mid\psi(-\iup\eta)=0\}=\{0\}$.
\end{theorem*}
It is also possible to relax the boundedness assumption on $f$ resulting in a \textbf{strong Liouville theorem}, cf.\ Theorem~\ref{gliou-35}.

The first part of our theorem complements and gives new insight to the result by Alibaud et al.~\cite{ali-et-al}. They provide necessary and sufficient conditions for the \eqref{vor-e14} in terms of the L\'evy triplet $(b,Q,\nu)$ appearing in \eqref{vor-e06} (see Section~\ref{sec-notes} for the precise statement), which we can identify to be equivalent to $\{\eta\in\rn\mid\psi(\eta)=0\}$. In fact, if we combine \cite{ali-et-al} and our criterion, it turns out that Alibaud et al.\ actually describe the orthogonal subgroup of the zero set of $\psi$, i.e.\ $\{\psi = 0\}^{\boxperp} := \{x\in\rn \mid \forall \gamma\in\{\psi=0\} \,:\, \eup^{\iup \gamma\cdot x}=1\}$. Let us point out that the approach of \cite{ali-et-al} is completely different from ours and, it seems that it is not possible to get a strong Liouville theorem using their methods. Thus, the second part \eqref{vor-e15} of our theorem is new.

We will give two independent proofs for Liouville's theorem, one is mainly analytic (see Section~\ref{sec-ana}) and another which is mainly probabilistic (see Section~\ref{sec-prob}). Both approaches have their (dis-)advantages. The analytic argument, which is based on the structure of generalized functions (distributions) in the sense of L.\ Schwartz, explains the appearance of the condition $\{\psi=0\}=\{0\}$ in a very natural way---but it seems to work only for smooth symbols. The more probabilistic argument is very short, but it requires a deep `black-box' theorem due to Choquet and Deny on certain convolution equations for measures. On the upside, the proof works in all dimensions and links Liouville's theorem with renewal theory. Section~\ref{sec-gliou} contains a version of Liouville's theorem where the function $f$ is not necessarily bounded and the concluding Section~\ref{sec-notes} contains a few further applications.

\section{General preparations}\label{sec-gen}

An important point is the question how we should interpret in a statement like
\begin{gather*}
    f\in L^\infty(\rn),\quad \Lcal_\psi f = 0 \implies f\equiv\text{const}
\end{gather*}
the expression $\Lcal_\psi f$ if $f$ is not in the (proper) domain of the operator $\Lcal_\psi$. A natural way is to understand $\Lcal_\psi f$ weakly in the sense of generalized functions, i.e.\ $\scalp{\Lcal_\psi f}{\phi}:= \scalp{f}{\Lcal^*_\psi\phi} = \scalp{f}{\Lcal_{\overline\psi}\phi}$ for all test functions $\phi\in C_c^\infty(\rn)$. The following lemma shows that this is always possible if $f\in L^\infty(\rn)$.

\begin{lemma}\label{gen-01}
    Let $\Lcal_\psi$ be the generator of a L\'evy process. For every $f\in L^\infty(\rn)$
    \begin{gather}\label{gen-e16}
        \phi \mapsto \scalp{\Lcal_\psi f}{\phi} := \scalp{f}{\Lcal_\psi^*\phi} = \int_\rn f(x) \Lcal_{\overline\psi}\phi(x)\,dx,
        \quad\phi\in C_c^\infty(\rn)
    \end{gather}
    is a Schwartz distribution $\Lcal_\psi f\in\Dcal'(\rn)$ of order $\leq 2$.
\end{lemma}
\begin{proof}\footnote{The lemma is also a direct consequence of the positive maximum principle: If $\phi\in \Scal(\rn)$ attains a positive maximum at $x_0$, then $\Lcal_\psi\phi(x_0)\leq 0$; see \cite[Theorem 6.2]{barca} and \cite[Theorem 2.21]{LM3}.}
Using Taylor's formula it is not hard to see that $\|\Lcal_\psi u\|_{L^p} \leq c\sum_{|\alpha|\leq 2} \|\partial^\alpha u\|_{L^p}$ for all $1\leq p\leq \infty$ and $u\in C_c^\infty(\rn)$, see e.g.~\cite[Lemma~3.4]{ieot}.
If we use $p=1$ and $\overline\psi$ instead of $\psi$, we conclude that
$\|\Lcal_\psi^*\phi\|_{L^1} = \|\Lcal_{\overline\psi}\phi\|_{L^1} \leq c\sum_{|\alpha|\leq 2} \|\partial^\alpha \phi\|_{L^1} \leq c_{K}\sum_{|\alpha|\leq 2} \|\partial^\alpha \phi\|_{L^\infty}$ for all $\phi\in C_c^\infty(K)$ and all compact sets $K\subset\rn$.
\end{proof}

In Section~\ref{sec-gliou} we will see further extensions of $\Lcal_\psi f$ to a class of functions $f$ which are not necessarily bounded.

\begin{remark}\label{gen-03}
    It is a classical observation, see e.g.\ \cite[Lemma 1.5, p. 233]{ali-et-al}, that the smoothness of $f$ is not essential. Since $\Lcal_\psi$ has constant `coefficients'---i.e.\ the triplet $(b,Q,\nu)$ appearing in \eqref{vor-e06} and \eqref{vor-e12} does not depend on $x$---it commutes with convolutions. If $j_\epsilon\in C_c^\infty(\rn)$ is the standard kernel for the Friedrichs mollifier and if the convolution $j_\epsilon*f$ is well-defined, we still have $\Lcal_\psi(j_\epsilon*f)=0$ and we infer from $j_\epsilon*f = \text{const}$ that $f$ is constant.
\end{remark}

We need the following results on the zero set of a continuous negative definite function.
\begin{lemma}\label{gen-05}
    Let $\psi:\rn\to\comp$ be the symbol of the generator of a L\'evy process. The following estimates hold
    \begin{gather}\label{gen-e18}
        \sqrt{|\psi(\xi+\eta)|} \leq \sqrt{|\psi(\xi)|} + \sqrt{|\psi(\eta)|},\quad\xi,\eta\in\rn;\\
    \label{gen-e20}
        |\psi(\xi)+\overline{\psi(\eta)}-\psi(\xi-\eta)|^{2} \leq 4|\psi(\xi)| |\psi(\eta)|,\quad\xi,\eta\in\rn;
    \end{gather}
    in particular, the zero-set $\{\psi=0\} = \{\xi\in\rn \mid \psi(\xi) = 0\}$ is a closed subgroup of $(\rn,+)$ which is either discrete or of the form $G\oplus E$ for some subspace $E\subset\rn$ and a closed discrete subgroup \textup{(}a relative lattice\textup{)} $G\subset E^\bot$.
    If $\{\psi=0\}\supsetneqq\{0\}$ is not trivial, $\psi$ must be periodic, hence the restriction of $\psi$ to $\spann(G)$ is bounded.
\end{lemma}
\begin{proof}
    The inequalities \eqref{gen-e18}, \eqref{gen-e20} are well-known, see Berg and Forst \cite[Proposition 7.15]{ber-for}, Jacob \cite[Lemma 3.621]{jac-1} or \cite[Theorem 6.2]{barca} for various proofs. Since $\psi(-\eta)=\overline{\psi(\eta)}$, the inequality \eqref{gen-e18} and the continuity of $\psi$ show that $\{\psi=0\}$ is a closed subgroup of $\rn$. The decomposition follows from a structure result on closed subgroups of $\rn$, see
    Bourbaki \cite[Chapter VII, \S 1.2, p.~72]{bourbaki}.  If $\eta\in\{\psi=0\}$ and $\eta\neq 0$, the estimate \eqref{gen-e20} shows that $\psi$ is periodic with period $\eta$. Since $\psi$ is continuous, it is bounded on the unit cell of the lattice, hence bounded on the subspace spanned by the lattice.
\end{proof}

\section{An analytic approach}\label{sec-ana}

We have seen in Lemma~\ref{gen-01} that $\Lcal_\psi f$ may always be understood as an element of $\mathcal D'(\rn)$ if $f\in L^\infty(\rn)$. On the other hand, the calculation \eqref{vor-e10} shows that $\Lcal_\psi u$ can only be a Schwartz distribution $\mathcal S'(\rn)$ if $\psi$ is a pointwise multiplier for $\Scal(\rn)$, i.e.\ a $C^\infty$-function which is polynomially bounded along with all of its derivatives. In this section we will first discuss this particular case before we move on to a general proof. The following lemma is also a special case of \cite[Lemma 4.3]{kuehn17}. In order to keep the presentation self-contained, we include a short argument.

\begin{lemma}\label{ana-08}
    Let $\psi:\rn\to\comp$ be a continuous negative definite function given by \eqref{vor-e06}. Then $\psi\in C^\infty(\rn)$ if, and only if, the L\'evy measure $\nu$ has finite moments $\int_{|x|\geq 1} |x|^k\,\nu(dx)$ of any order $k\in\nat_0$. If this is the case, $\psi$ and all of its derivatives are polynomially bounded.
\end{lemma}
\begin{proof}
    Without loss of generality we may assume that the representation of \eqref{vor-e06} $\psi$ consists only of the integral part.

    Note that the integrand appearing in \eqref{vor-e06} satisfies for any multi-index $\alpha\in\nat_0^n$,
    \begin{gather*}
        \partial_\xi^\alpha\left(1-\eup^{\iup \xi\cdot x} + \iup\xi\cdot x \I_{(0,1)}(|x|)\right)
        =
        \begin{cases}
        1-\eup^{\iup \xi\cdot x} + \iup\xi\cdot x \I_{(0,1)}(|x|), & |\alpha|=0\\
        -\iup^{|\alpha|} x^\alpha \eup^{\iup \xi\cdot x} + \iup x^\alpha\I_{(0,1)}(|x|), & |\alpha|=1\\
        -\iup^{|\alpha|} x^\alpha \eup^{\iup \xi\cdot x}, & |\alpha|\geq 2.
        \end{cases}
    \end{gather*}
    These expressions are bounded, up to a multiplicative constant, by $|x|^2(1+|\xi|^2)$ (if $|\alpha|=0$), $|x|^{|\alpha|+1}(1+|\xi|)$ (if $|\alpha|=1$) or $|x|^{|\alpha|}$ (if $|\alpha|\geq 2$), respectively.

    A routine application of the differentiability lemma for parameter-dependent integrals shows that $\int_{x\neq 0}|x|^{|\alpha|}\,\nu(dy)<\infty$ ensures that the derivatives up to order $|\alpha|\geq 2$ exist (and are polynomially bounded). Since $\int_{|x|\leq 1}|x|^2\,\nu(dx)<\infty$ is always satisfied, the condition becomes $\int_{|x|\geq 1}|x|^{|\alpha|}\,\nu(dx)<\infty$.

    In order to show the converse, we note that $\psi$ is smooth if, and only if, $\Re\psi$ and $\Im\psi$ are smooth. Define $\phi_\xi(t):=\Re\psi(t\xi)$, $t\in\real$. Since $\Re\psi$ is differentiable, the mean value theorem shows that for some $\theta=\theta(t)\in (0,1)$
    \begin{gather*}
        0 \leq \int_{\rn\setminus\{0\}} \frac{1-\cos(t\xi\cdot y)}{t^2}\,\nu(dy)
        = \frac{\phi_\xi(t)-\phi_\xi(0)}{t^2}
        = \frac{\phi_\xi'(t\theta)}{t}
        \leq \frac{\phi_\xi'(t\theta)}{t\theta}.
    \end{gather*}
    Since $\phi_\xi(\cdot)$ is an even function, we have $\phi_\xi'(0)=0$, and we see that the right-hand side converges to $\phi_\xi''(0)$ as $t\to 0$. Using Fatou's lemma on the left-hand side reveals
    \begin{gather*}
        \frac 12\int_{\rn\setminus\{0\}} (\xi\cdot y)^2\,\nu(dy)
        \leq \liminf_{t\to 0} \int_{\rn\setminus\{0\}} \frac{1-\cos(t\xi\cdot y)}{t^2}\,\nu(dy)
        \leq \phi_\xi''(0).
    \end{gather*}
    This gives $\int_{\rn\setminus\{0\}} |y|^2\,\nu(dy)<\infty$. Higher derivatives can be dealt with in a similar fashion using induction over $\phi^{(2k)}_{\xi}$, the $2k$-th derivative of $\phi_{\xi}$.

    The first part of the proof shows that all derivatives $\partial^\alpha\psi(\xi)$ are bounded by $c(1+|\xi|^2)$ (if $|\alpha|=0$) or  $c(1+|\xi|)$ (if $|\alpha|=1$) or a constant (if $|\alpha|\geq 2$).
\end{proof}

We can now state and prove our main theorem for smooth symbols.
\begin{theorem}[Liouville]\label{ana-10}
    Let $\Lcal_\psi$ be the generator of a L\'evy process with characteristic exponent $\psi$ and assume that $\psi$ is a $C^\infty$-function. The operator $\Lcal_\psi$ has the Liouville property \eqref{vor-e14} if, and only if, the zero-set of the symbol satisfies $\{\psi=0\}=\{0\}$.
\end{theorem}
\begin{proof}
    Let $f\in L^\infty(\rn)$ and $\Lcal_\psi f = 0$ in $\Dcal'(\rn)$. In view of Lemma~\ref{gen-05} we can assume that $\{\psi=0\}$ is a discrete group, otherwise we would have a truly lower-dimensional problem. Since $\psi$ is smooth, it is a pointwise multiplier in $\Scal(\rn)$, see Lemma~\ref{ana-08}, and we conclude that in $\Scal'(\rn)$
    \begin{gather*}
        0
        = \scalp{\Lcal_\psi f}{\phi}
        = \scalp{f}{\Lcal_{\overline\psi} \phi}
        = \scalp{\widecheck f}{\widehat{\Lcal_{\overline\psi} \phi}}
        = \scalp{\widecheck f}{\overline\psi \widehat{\phi}}
    \end{gather*}
    holds for all $\phi\in\Scal(\rn)$ and all $\widehat\phi\in\Scal(\rn)$ (as the Fourier transform is bijective on $\Scal(\rn)$). We conclude that $\supp\widecheck f\subset \{\overline\psi=0\}=\{\psi=0\}$.

    We can now use a classical result on the structure of tempered distributions supported in single points, see e.g.\
    Tr\`eves \cite[Chapter 24]{treves},
    \begin{gather*}
        \smash[b]{\widecheck f = \sum_{\gamma\in\{\psi=0\}} \sum_{|\alpha|\leq n(\gamma)<\infty} c_{\alpha,\gamma} \partial^\alpha\delta_\gamma}
    \intertext{or, equivalently,}
        f(x) = \sum_{\gamma\in\{\psi=0\}} \sum_{|\alpha|\leq n(\gamma)<\infty} (2\pi)^{-n}c_{\alpha,\gamma} (\iup x)^\alpha \eup^{-\iup\gamma\cdot x}.
    \end{gather*}
    If, in addition, $f$ is bounded, we have $f(x) = (2\pi)^{-n}\sum_{\gamma\in\{\psi = 0\}} c_{0,\gamma} \eup^{-\iup\gamma\cdot x}$.

    If $\{\psi=0\}=\{0\}$, it is clear that the Liouville property holds. Conversely, if the Liouville property holds and if $\{\psi=0\}$ is not trivial, we can always shift a solution $\Lcal_\psi f=0$ in the following way: $f_\eta(x):= f(x)\eup^{-\iup\eta\cdot x}$, $\eta\in\{\psi=0\}$ and we get again $\Lcal_\psi f_\eta = 0$; this means that can always achieve that some $c_{0,\gamma}\neq 0$, $\gamma\neq 0$, and we have reached a contradiction to $f_\eta\equiv\text{const}$. Thus, the above representation shows that the Liouville property holds if, and only if, $\{\psi = 0\}$ is trivial. If $\{\psi = 0\}\supsetneqq\{0\}$, the solutions to $\Lcal_\psi f = 0$ are periodic with periodicity group given by the orthogonal subgroup (cf.\ \cite[Definition~2.8]{ber-for}) of the zero-set $\{\psi = 0\}^{\boxperp} := \{x\in\rn \mid \forall \gamma\in\{\psi=0\} \,:\, \eup^{\iup \gamma\cdot x}=1\}$.
\end{proof}

\begin{remark}\label{ana-13}
\textbf{a)\ }
    Although one can approximate (locally uniformly) any symbol $\psi$ with smooth symbols $\psi_k$---just cut off the L\'evy measure $\nu_k(dx) := \I_{B_k(0)}(x)\,\nu(dx)$---it seems to be difficult to obtain a general version of Theorem~\ref{ana-10} through a limiting argument.

\textbf{b)\ }
    It is possible to construct symbols on $\real$ which are continuous but nowhere differentiable. Consider, for example a variation of Weierstra{\ss}'s nowhere differentiable function (see \cite[Theorem 1.31, p. 303]{hardy})
    \begin{gather*}
        \psi(\xi) := \sum_{k=0}^\infty a^k\left(1-\cos(b^k\pi \xi)\right),
        \quad\xi\in\real,\quad a\in (0,1),\; b>1 \text{\ \ and\ \ } ab\geq 1
    \end{gather*}
 which is the characteristic exponent of a L\'evy process with L\'evy measure $\sum\limits_{k=0}^\infty a^k\,\delta_{b^k\pi}(dx)$.
\end{remark}

\section{A probabilistic proof}\label{sec-prob}

Our starting point is a theorem due to Choquet and Deny \cite{cho-den60}; full proofs are given in \cite{deny60}, a probabilistic (martingale) argument is due to Doob, Snell and Williamson \cite{dsw60}. Recall that the \textbf{support} of a measure $\mu$ is the complement of the union of all open $\mu$ null sets, i.e.\  $(\supp\mu)^c = \bigcup_{\{U\subset\rn \text{\ open} \,\mid\, \mu(U)=0\}}U$. This means that $x\in\supp\mu$ if, and only if, there is an open neighbourhood $U=U(x)$ with $\mu(U)>0$.

\begin{theorem}[Choquet--Deny]\label{prob-15}
    Let $\mu$ be a probability measure on $\rn$ and $h$ a bounded and continuous function. One has $h = h*\mu$ if, and only if, every point of the support of $\mu$ is a period of $h$.
\end{theorem}

In order to apply Theorem~\ref{prob-15} to the Liouville problem, we need to reduce the problem to an assertion on probability measures. Since $\Lcal_\psi$ is the generator of a L\'evy process $(X_t)_{t\geq 0}$, the transition probabilities $\mu_t(dx):=\Pp(X_t\in dx)$, $t\geq 0$, form a convolution semigroup of probability measures. We write $\Pcal_t u(x):= u*\widetilde{\mu}_t(x) = \Ee u(X_t+x)$ ($\widetilde\mu_t(dy)=\mu_t(-dy)$) for the Markov semigroup associated with $(X_t)_{t\geq 0}$. Since $t\mapsto X_t$ is right-continuous, a straightforward application of the dominated convergence theorem shows that $t\mapsto \Pcal_t u(x)$ and $x\mapsto\Pcal_t u(x)$ are right-continuous, resp.\ continuous if $u\in C_b(\rn)$.
\begin{lemma}\label{prob-17}
    Let $\Lcal_\psi$ be the generator of a L\'evy process and denote by $(\Pcal_t)_{t\geq 0}$ and $(\Rcal_\lambda)_{\lambda>0}$ the corresponding semigroup and resolvent family, respectively, and let $f\in C_b(\rn)$. The following assertions are equivalent:
    \begin{enumerate}
    \item\label{prob-17-a} $\Lcal_\psi f=0$ weakly, i.e.\ $\scalp{\Lcal_\psi f}{\phi} = 0$ for all $\phi\in C_c^\infty(\rn)$.
    \item\label{prob-17-b} $\lambda\Rcal_\lambda f = f$ for all $\lambda>0$.
    \item\label{prob-17-c} $\Pcal_t f = f$ for all $t>0$.
    \end{enumerate}
\end{lemma}
\begin{proof}
    Since $f$ is bounded, we may assume that $f$ is positive, otherwise we could consider $f+\|f\|_\infty$. We are going to show the implications \ref{prob-17-a}$\Rightarrow$\ref{prob-17-b}$\Rightarrow$\ref{prob-17-c}$\Rightarrow$\ref{prob-17-a}.

    Assume that $\Lcal_\psi f=0$ weakly for a function $f\in C_b(\rn)$. Since both $\Lcal_\psi$ and $\Lcal^*_\psi = \Lcal_{\overline\psi}$ generate strongly continuous contraction semigroups in $L^1(\rn)$, see \cite[Proposition 12.7]{ber-for}, we know from the Hille--Yosida theorem that $\Rcal^*_\lambda$ maps $L^1(\rn)$ into $D(\Lcal_\psi^*)$. Using the fact that the test functions $C_c^\infty(\rn)$ are an operator core
    for the $L^1$-generator $(\Lcal_\psi^*,D(\Lcal_\psi^*)) = (\Lcal_{\overline\psi},D(\Lcal_{\overline\psi}))$, we can find for every $\phi\in\Scal(\rn)$ a sequence $(u_n)_{n\in\nat}\subset C_c^\infty(\rn)$ such that both $u_n \to \Rcal^*_\lambda\phi$ and $\Lcal_\psi^* u_n \to \Lcal_\psi^* \Rcal^*_\lambda\phi$ in $L^1$. Therefore,
    \begin{gather*}
        \scalp{\Rcal_\lambda \Lcal_\psi f}{\phi}
        =\scalp{\Lcal_\psi f}{\Rcal_\lambda^* \phi}
        =\lim_{n\to\infty}\scalp{\Lcal_\psi f}{u_n}.
    \end{gather*}
    By assumption, the right-hand side is zero, i.e.\ it holds that $\lambda\Rcal_\lambda f -f= 0$ for all $\lambda>0$, and \ref{prob-17-b} follows.

    Since $\lambda \Rcal_\lambda f(x) = \int_0^\infty \lambda \eup^{-\lambda t} (\Pcal_t f)(x)\,dt$ we get for fixed $x$ that
    \begin{gather*}
        \lambda\Rcal_\lambda f(x) = f(x)
        \iff
        \int_0^\infty \eup^{-\lambda t} (\Pcal_t f)(x)\,dt = \frac{f(x)}{\lambda} = \int_0^\infty \eup^{-\lambda t} f(x)\,dt.
    \end{gather*}
 Because of the uniqueness of the Laplace transform (for positive measures), we infer that $\Pcal_t f(x) = f(x)$; here we use the right-continuity of $t\mapsto\Pcal_t f(x)$.

    Since $\Pcal_t^*$ is a strongly continuous semigroup on $L^1(\rn)$ and $\Scal(\rn)\subset D(\Lcal_\psi^*)$, we see that \ref{prob-17-c} implies
    \begin{gather*}
        0 = \frac 1t\scalp{\Pcal_t f-f}{\phi}
        =\scalp{f}{t^{-1}(\Pcal_t^* \phi-\phi)}
        \xrightarrow[t\to 0]{}\scalp{f}{\Lcal_\psi^*\phi}
    \end{gather*}
    for any $\phi\in\Scal(\rn)$, and we get \ref{prob-17-a}.
\end{proof}

Finally, we will need the following characterization of multivariate lattice distributions. Recall that a probability measure $\mu$ is said to be a \textbf{lattice distribution} in $\rn$ if $\supp\mu$ is a lattice in $\rn$.
\begin{lemma}\label{prob-19}
    A probability measure $\mu$ on $\rn$ is a lattice distribution if, and only if, there are linearly independent vectors $\xi^1,\dots,\xi^n\in\rn$ such that the characteristic function satisfies $|\widecheck\mu(\xi^j)|=1$ for all $j=1,\dots,n$.
\end{lemma}
\begin{proof}
    In one dimension this is a classic result, see e.g.\ Lukacs \cite[Theorem 2.1.4., pp.\ 17--18]{lukacs70}.

    If $n>1$ we can mimic the one-dimensional proof. Without loss of generality, we may assume that $\supp\mu$ is a proper lattice, i.e.\ it is of the form $\theta^0 + \sum_{j=1}^n l_j\theta^j$ with $l=(l_1,\dots,l_n)\in\integer^n$, $\theta^0\in\rn$, and a basis $\theta^1,\dots,\theta^n$ of $\rn$. Up to a linear transformation, we may even assume that the basis is orthonormal.

    If $\mu$ is a (proper) lattice distribution, it is a discrete distribution of the form $\mu = \sum_{l\in\integer^n} p_l \delta_{\Theta_l}$ with $\Theta_l = \theta^0 + \sum_{j=1}^n l_j\theta^j$, $p_l\geq 0$ and $\sum_{l\in\integer^n}p_l=1$. Thus, the characteristic function is given by
    \begin{gather*}
        \widecheck\mu(\xi)
        = \sum_{l\in\integer^d} p_l \exp\left[\iup \xi\cdot\Theta_l\right]
        = \exp\left[\iup \xi\cdot\theta^0\right]\sum_{l\in\integer^d} p_l \exp\left[\iup l_1\xi\cdot\theta^1+\dots + \iup l_n\xi\cdot\theta^n\right]
    \end{gather*}
    and the claim follows upon taking $\xi^j = 2\pi|\theta^j|^{-2}\theta^j$, $j=1,\dots,n$.

    Conversely, assume that $|\widecheck\mu(\xi^j)|=1$, that is $\widecheck\mu(\xi^j)=\eup^{\iup \alpha_j}$ for suitable $\alpha_j\in\real$. Then
    \begin{gather*}
        1
        = \int_{\rn} \exp\left[\iup (\xi^j\cdot x - \alpha_j)\right]\mu(dx)
        = \int_{\rn} \cos\left(\xi^j\cdot x - \alpha_j\right)\mu(dx)
    \end{gather*}
    and, since $\int_\rn 1\,d\mu = 1$, we infer that
    \begin{gather*}
        \int_{\rn} \left(1-\cos(\xi^j\cdot x - \alpha_j)\right)\mu(dx)=0.
    \end{gather*}
    Since the integrand is positive, we see that $\xi^j\cdot x - \alpha_j\in 2\pi\integer$ showing that $\supp\mu$ is contained in the lattice spanned by $(\xi^1,\dots,\xi^n)$ and shifted by $\alpha=(\alpha_1,\dots,\alpha_n)$.
\end{proof}

We are now ready for the proof of
\begin{theorem}[Liouville]\label{prob-21}
    Let $\Lcal_\psi$ be the generator of a L\'evy process with characteristic exponent $\psi$. The operator $\Lcal_\psi$ has the Liouville property \eqref{vor-e14} if, and only if, the zero-set of the characteristic exponent satisfies $\{\psi=0\}=\{0\}$.
\end{theorem}
\begin{proof}
Let $f:\rn\to\rn$ be a bounded, non-constant weak solution of $\Lcal_\psi f=0$. The argument used in Remark~\ref{gen-03} allows us to assume that $f\in C^\infty_b(\rn)$. Thus, we can use Lemma~\ref{prob-17} to see that $\Lcal_\psi f=0$ is equivalent to $\Pcal_t f=f$ for all $t>0$.

As $f$ is continuous, we know that
\begin{align*}
    \Per(f):=\{x\in\rn \mid f(x+y)=f(y)\text{\ for all\ }y\in\rn\}
\end{align*}
is a closed subgroup of the additive group $(\rn,+)$. An application of the Choquet--Deny theorem, Theorem~\ref{prob-15}, reveals that
\begin{align}\label{prob-e22}
    \bigcup_{t>0}\supp(\mu_t)\subset \Per(f).
\end{align}
Since $\Per(f)$ is a closed additive subgroup of $\rn$, it is either discrete or of the form $G\oplus E$ for some subspace $E\subset\rn$ and a closed subgroup (a relative lattice) $G\subset E^\perp$, see \cite[Chapter VII, \S 1.2, p.~72]{bourbaki} or Lemma~\ref{gen-05}. By a linear isomorphism we may assume that $\Per(f)=\real^r\times \integer^{n-r}$. If $r>0$, everything is reduced to a lower-dimensional problem, as $f$ is constant on every affine subspace $\real^r\times \{z\}$ for each $z\in \real^{n-r}$---the new characteristic function is $\widecheck{\mu}_t((0,\dotso,0,\xi_1,\dotso,\xi_{n-r}))$. Therefore, we may assume that $\Per(f)=\integer^n$ which implies that $\mu_t$ is for every $t>0$ a non-trivial lattice distribution on $\integer^n$. From the proof of Lemma~\ref{prob-19} we know that  $\widecheck\mu_t(2\pi e_1) = \widecheck\mu_t(2\pi e_2) = \dots = \widecheck\mu_t(2\pi e_n)=1$, where $e_j=( \delta_{ij})_{i\in \{1,\dotso,n\}}$. Since $\widecheck\mu_t(\xi) = \eup^{-t\psi(\xi)}$ and since $t>0$ is arbitrary, this means that $2\pi e_j\in\{\psi=0\}$ for all $j=1,2,\dots,n$, and we conclude that $\{\psi=0\}\neq\{0\}$.

\bigskip
Conversely assume that $\{\psi=0\} \neq \{0\}$. Without loss of generality we may assume that $\eta=(2\pi,0\dots,0)\in \{\psi=0\}$. Denote by $\pi_1 : (\xi_1,\dots,\xi_n)\mapsto \xi_1$ the projection onto the first coordinate. Since the image measure $\mu_t^{(1)} := \mu_t\circ\pi_1^{-1}$ has the characteristic function $\widecheck\mu_t^{(1)}(\xi) = \widecheck\mu_t((\xi_1,0,\dots,0)) = \eup^{-t\psi(\xi_1,0,\dots,0)}$ for all $t>0$, the distribution $\mu_1$ is a lattice distribution and it is clear that $f(x):=\sin(2\pi x_1)$ is a solution of $\Pcal_t f=f$.
\end{proof}
Let us point out that the condition in Theorem \ref{prob-21} is equivalent to the fact that $(x,y)\mapsto\sqrt{|\psi(x-y)|}$ defines a distance on $\real^n$.
\section{Towards the strong Liouville property}\label{sec-gliou}

For the Laplace operator \eqref{vor-e02} remains valid if we assume $f\geq 0$ rather than $f\in L^\infty(\rn)$. This is the so-called \textbf{strong Liouville property}. In this section we want to discuss how we can relax the condition $f\in L^\infty(\rn)$ in the corresponding problem \eqref{vor-e14} for L\'evy generators.

The key to such a result is the following Choquet representation theorem which replaces Theorem~\ref{prob-15}, see \cite{deny60}. Here, and in the sequel, we understand that integrals of positive measurable functions always exist in $[0,+\infty]$.
\begin{theorem}[Deny]\label{gliou-22}
    Let $\mu$ be a probability measure on $\rn$ such that the smallest group containing the support of $\mu$ is $\rn$. All positive solutions of the convolution equation $h = h*\mu$ are of the form
    \begin{gather*}
        h(x) = \int_{E(\mu)} \eup^{-\xi\cdot x}\,\kappa(d\xi)
    \end{gather*}
    where $\kappa$ is a unique \textup{(}positive\textup{)} measure with support in the set of the $\mu$-harmonic exponentials $E(\mu) = \left\{\xi\in\rn \mid \int_\rn \eup^{\xi\cdot y}\,\mu(dy) =1\right\}$.
\end{theorem}

A close inspection of our arguments in Section~\ref{sec-prob} reveals that we use boundedness of $f$ for essentially two purposes:
a) in order to make sense of the weak formulation $\scalp{\Lcal_\psi f}{\phi} = \scalp{f}{\Lcal_\psi^* \phi} = \int f(x)\Lcal_{\psi^*}\phi(x)\,dx$ (which requires that $f\in L^\infty$ as we only know that $\Lcal_\psi^* \phi\in L^1$)
and
b) to make sense of $\Pcal_t f(x) = \Ee f(X_t+x)$ (mainly in Lemma~\ref{prob-17} where we reduce the Liouville property to a property of the semigroup). Below we will provide arguments which allow us to get rid of these restrictions if $f$ satisfies a growth bound.  Our results complement previous work by K\"{u}hn \cite{kuehn20} who looks at not necessarily positive solutions $f$ which have polynomial growth.

\begin{definition}\label{gliou-23}
    A measurable function $g:\rn\to (0,\infty)$ is said to be \textbf{submultiplicative}, if there exists some constant $c>0$ such that
    \begin{align*}
        g(x+y)\leq cg(x)g(y)\quad\text{for all\ \ } x,y\in\rn.
    \end{align*}
\end{definition}
It is well-known that a locally bounded submultiplicative function is exponentially bounded, i.e.\ $g(x)\leq \alpha\exp(\beta|x|)$ for suitable constants $\alpha,\beta > 0$, see \cite[Lemma 25.5]{sato}. Moreover, submultiplicative functions appear naturally in connection with the existence of generalized moments of L\'evy processes. If $(X_t)_{t\geq 0}$ is a L\'evy process with L\'evy triplet $(b,Q,\nu)$ and $g$ a locally bounded submultiplicative function, then
\begin{gather}\label{gliou-e23}
    \text{ for some (or all) $t$\::}\:\:\Ee g(X_t)<\infty
    \quad\text{if, and only if,}\quad
    \int_{|y|\geq 1} g(y)\,\nu(dy)<\infty,
\end{gather}
cf.\ \cite[Theorem 25.3]{sato}. This moment condition allows us to extend $\scalp{\Lcal_\psi f}{\phi}$ to functions $f$ which are dominated by a locally bounded submultiplicative function.

\begin{lemma}\label{gliou-25}
    Let $\Lcal_\psi$ be the generator of a L\'evy process with characteristic exponent $\psi$ given by \eqref{vor-e06} and let $g:\rn\to (0,\infty)$ be a locally bounded submultiplicative function. Then we have
    \begin{align}\label{gliou-e24}
        \int_{|y|\geq 1}g(y)\,\nu(dy)<\infty
        \quad\text{if, and only if,}\quad
        \forall\phi\in C_c^\infty(\rn)\::\: \|g\Lcal_{\overline\psi}\phi\|_{L^1}<\infty.
    \end{align}
\end{lemma}
\begin{proof}
    We split the exponent $\psi$ into two parts,
    \begin{gather*}
        \psi_1(\xi) := \int_{|y|\geq 1} \left(1-\eup^{\iup \xi\cdot y}\right)\nu(dy)
        \quad\text{and}\quad
        \psi_2(\xi) := \psi(\xi)-\psi_1(\xi).
    \end{gather*}
    Both $\psi_1$ and $\psi_2$ are again characteristic exponents of L\'evy processes whose triplets are $(0,0,\nu_1)$, $\nu_1(dy) := \I_{[1,\infty)}(|y|)\nu(dy)$ and $(b,Q,\nu_2)$, $\nu_2(dy) := \I_{(0,1)}(|y|)\nu(dy)$, respectively.

    Since the L\'evy measure appearing in the definition of $\Lcal_{\overline\psi_2}$ is supported in $\overline{B_1(0)}$, it is easy to see from  the integro-differential representation \eqref{vor-e12} of the generator that $\supp\big(\Lcal_{\overline\psi_2}\phi\big) \subset \supp\phi+ \overline{B_1(0)}$, i.e.\ $\Lcal_{\overline\psi_1}$ maps $C_c^\infty(\rn)$ into itself. In particular, $g\Lcal_{\overline\psi}\phi\in L^1(dx)$ for every $\phi\in C_c^\infty(\rn)$, which means that we can assume that $\Lcal_{\overline\psi}=\Lcal_{\overline\psi_1}$.

    Notice that
    \begin{align*}
         |g(x)\Lcal_{\overline\psi_1}\phi(x)|
        &= \bigg|\int \left[\phi(x-y)g(x-y)-\phi(x)g(x)\right] \nu_1(dy) \\
        &\quad\mbox{}+ \int \phi(x-y) \left[g(x)-g(x-y)\right] \nu_1(dy)\bigg|\\
	&\leq \int|\phi(x-y)g(x-y)|\nu_1(dy)+ |g(x)\phi(x)|\nu_1(\rn)\\
	&\quad\mbox{}+ \int |\phi(x-y)|\ |g(x)-g(x-y)| \nu_1(dy).
    \end{align*}
    By submultiplicativity  and the fact that $g$ is positive, we get
    \begin{align*}
        |g(x-y)-g(x)|
        &= |g(x-y)-g(x-y+y)|\\
        &\leq    g(x-y) + g((x-y)+y)
        \leq g(x-y)(1+c g(y)),
    \end{align*}
    and so, since $\nu_1(\rn) = \nu(|y|\geq 1) = \|\nu_1\|<\infty$,
    \begin{align*}
        \|g\Lcal_{\overline\psi_1}\phi\|_{L^1}
        &\leq 2\|\phi g\|_{L^1} \cdot \|\nu_1\| + \iint |\phi(x-y)g(x-y)| \left(1+c g(y)\right) \nu_1(dy)\,dx\\
        &= 3\|\phi g\|_{L^1} \cdot \|\nu_1\| + c\|\phi g\|_{L^1} \int g(y)\, \nu_1(dy).
    \end{align*}
    This proves sufficiency in \eqref{gliou-e24}.

    In order to get necessity, observe that the function $1+g$ is again submultiplicative. Since $\|g\mathcal{L}_{\overline\psi}\phi\|_{L^1}$ and $\|(1+g)\mathcal{L}_{\overline\psi}\phi\|_{L^1}$ are at the same time finite or infinite, we can assume that $g(x)\ge 1$ for all $x\in\rn$. Pick $\phi\in C_c^\infty(\rn)$ such that $\I_{B_1(0)}\leq\phi\leq \I_{B_2(0)}$. As in the first part of the proof we can assume that $\Lcal_{\overline\psi}=\Lcal_{\overline\psi_1}$. We have
    \begin{align*}
        \int |g(x)\Lcal_{\overline\psi_1}\phi(x)|\,dx
        &\geq \int_{|x|\geq 2} \left|g(x) \int_{|y|\geq 1} \left[\phi(x-y)-\phi(x)\right] \nu(dy)\right| dx\\
        &= \int_{|x|\geq 2} \int_{|y|\geq 1} g(x)\phi(x-y)\, \nu(dy)\, dx.
    \intertext{Since $g$ is submultiplicative, we see that $c g(y-x) g(x)\geq g(y)$. Thus,}
        \int |g(x)\Lcal_{\overline\psi_1}\phi(x)|\,dx
        &\geq \frac 1c \int_{|y|\geq 1} \int_{|x|\geq 2} \frac{\phi(x-y)}{g( y-x)}\,dx \, g(y)\, \nu(dy)\\
        &\geq \frac 1c \int_{|y|\geq 3} \int_{|x-y|\leq 1} \frac{\phi(x-y)}{g( y-x)}\,dx \, g(y)\, \nu(dy)\\
        &= \frac 1c \int_{|z|\leq 1} \frac{dz}{g(z)}  \int_{|y|\geq 3} g(y)\, \nu(dy)
    \end{align*}
    finishing the proof.
\end{proof}

\begin{lemma}\label{gliou-27}
    Let $(X_t)_{t\in[0,\infty)}$ be a L\'evy process with characteristic function $\psi$ given by \eqref{vor-e06} and let $g:\rn\to (0,\infty)$ be a locally bounded submultiplicative function. For every constant $C\in (0,\infty)$, the family $(g(X_t))_{t\in [0,C]}$ is uniformly integrable if, and only if, $\int_{|y|\geq 1}g(y)\nu(dy)<\infty$.
\end{lemma}
\begin{proof}
    As in the proof of Lemma~\ref{gliou-25} we write $\psi = \psi_1 + \psi_2$ where $\psi_1, \psi_2$ are characteristic exponents of L\'evy processes $X^1  = (X_t^1)_{t\geq 0}$ and $X^2 = (X_t^2)_{t\geq 0}$ with triplets $(0,0,\nu_1)$, $\nu_1(dy) := \I_{[1,\infty)}(|y|)\nu(dy)$ and $(b,Q,\nu_2)$, $\nu_2(dy) := \I_{(0,1)}(|y|)\nu(dy)$, respectively.

    It is well known, cf.\ \cite{sato}, that $X_t = X_t^1 + X_t^2$ and that the processes $X^1$ and $X^2$ are stochastically independent. Because of the submultiplicativity of $g$,
    \begin{align*}
    \Ee& \left[g(X_t^1+X_t^2)\I_{\left\{g(X_t^1+X_t^2)>r\right\}}\right]\\
    &\leq c \Ee\left[ g(X_t^1)g(X_t^2)\I_{\left\{g(X_t^1)g(X_t^2)>r/c\right\}}\right]\\
    &\leq \Ee \left[g(X_t^1)g(X_t^2)\left(\I_{\left\{g(X_t^1)>\sqrt{r/c}\right\}} + \I_{\left\{g(X_t^2)>\sqrt{r/c}\right\}}\right)\right]\\
    &= \Ee \left[g(X_t^1)\I_{\left\{g(X_t^1)>\sqrt{r/c}\right\}}\right] \Ee\left[ g(X_t^2)\right] + \Ee \left[g(X_t^2)\I_{\left\{g(X_t^2)>\sqrt{r/c}\right\}}\right] \Ee \left[g(X_t^1)\right].
    \end{align*}
    Any locally bounded submultiplicative function is exponentially bounded, and so we have $g(x)\leq \alpha\exp(\beta|x|)$. Since a L\'evy process with bounded jumps has exponential moments, cf.~\cite[Theorem 25.3, p.~159]{sato}, we see that
    \begin{align*}
        \Ee \left[g^{ 2}(X_t^2)\right]
        \leq \alpha\Ee\left[ \eup^{ 2\beta |X_t^2|}\right]
        \leq \alpha\sup_{t\leq C}\Ee\left[ \eup^{2\beta |X_t^2|}\right]
        < \infty.
    \end{align*}
    This implies that $(g(X_t^2))_{t\in [0,C]}$ is  $L^2$-bounded, hence uniformly integrable.

    In order to prove uniform integrability of $(g(X_t^1))_{t\in [0,C]}$, we observe that $X_t^1$ is a compound Poisson process. This is in distribution equal to $\sum\limits_{i=0}^{N_t}Y_i$ where $N_t$ is an independent Poisson process with parameter $\nu_1(\rn)$, and $(Y_i)_{i\in \nat}$ is an independent (of $N_t$) sequence of independent and identically distributed random variables, such that each $Y_i$ has the distribution $\nu_1/\nu_1(\rn)$. The sum $\sum_{i=0}^{N_t} Y_i$ has the probability law
    \begin{align*}
    P(N_t=0)\delta_0(dy) + \sum_{k=1}^\infty P(N_t=k) P(Y_1+\dots+Y_k\in dy)
    &= \sum_{k=0}^\infty \frac{P(N_t=k)}{\nu_1(\rn)^k} \,\nu_1^{*k}(dy)\\
    &= \sum_{k=0}^\infty \eup^{-t\nu_1(\rn)} \frac{t^k}{k!} \,\nu_1^{*k}(dy).
    \end{align*}
     ($\nu_1^{*k}$ denotes the $k$-fold convolution product, $\nu_1^{*0}=\delta_0$). A good reference is \cite[Chapter 1.4]{sato} or \cite[Example 3.2.d), Theorem 3.4]{barca}. We conclude that
    \begin{align*}
        \Ee \left[g(X_t^1)\I_{\left\{g(X_t^1)>\sqrt{r/c}\right\}}\right]
        &= \sum_{k=0}^\infty \eup^{-t\nu_1(\rn)}\frac{t^k}{k!}\int_{\rn} g(y) \I_{\left\{g(y)>\sqrt{r/c}\right\}}\,\nu_1^{\ast k}(dy)\\
        &\leq \sum_{k=0}^\infty \frac{C^k}{k!}\int_{\rn} g(y)\I_{\left\{g(y)>\sqrt{r/c}\right\}}\,\nu_1^{\ast k}(dy).
    \end{align*}
    We can now use Lebesgue's dominated convergence theorem to get
    \begin{align*}
        \lim_{r\to\infty}
        \sum_{k=0}^\infty \frac{C^k}{k!} \int_{\rn}g(y) \I_{\left\{g(y)>\sqrt{r/c}\right\}}\nu_1^{\ast k}(dy)
        = 0.
    \end{align*}
    This shows that $(g(X_t^1))_{t\in [0,C]}$ is uniformly integrable, finishing the proof.
\end{proof}

We can now show an analogue of Lemma~\ref{prob-17} for unbounded $f$.
\begin{lemma}\label{gliou-29}
    Let $\Lcal_\psi$ be the generator of a L\'evy process with characteristic function $\psi$ given by \eqref{vor-e06} and semigroup $(\Pcal_t)_{t\geq 0}$. Assume that $g:\rn\to (0,\infty)$ is a locally bounded submultiplicative function satisfying $\int_{|y|\geq 1} g(y)\,\nu(dy)<\infty$. For any $f\in C(\rn)$ such that $|f(x)|\leq g(x)$ the following assertions are equivalent:
    \begin{enumerate}
    \item\label{gliou-29-a} $\Lcal_\psi f=0$ weakly, i.e.\ $\scalp{\Lcal_\psi f}{\phi} = 0$ for all $\phi\in C_c^\infty(\rn)$;
    \item\label{gliou-29-b} $\Pcal_t f = f$ for all $t>0$.
    \end{enumerate}
\end{lemma}

\begin{proof}
An obvious variation of the proof of Lemma~\ref{gliou-27} reveals that also the family $(g(X_t+z))_{t\in [0,C], |z|\leq R}$ is uniformly integrable.

Since $|f(x)|\leq g(x)$, we see that ${(f(X_t+z))_{t\in [0,C],|z|\leq R}}$ is also uniformly integrable and we can use Vitali's convergence theorem (e.g.\ \cite[Theorem 22.7, p.~262]{mims}) to conclude that $x\mapsto\Pcal_t f(x) = \Ee f(X_t+x)$ is continuous.

Let $\phi\in C^\infty_c(\rn)$. We see with Dynkin's formula \cite[Lemma 12.2]{barca} or \cite[p.~27]{LM3} that
\begin{align}
    \int_{\rn}\left[\Pcal_tf(x) - f(x)\right] \phi(x)\, dx
    &=\notag \int_{\rn} f(x)\left[\Pcal^*_t\phi(x)-\phi(x)\right] dx\\
    &=\notag \int_{\rn} f(x)\int_{0}^t \Pcal^*_s \Lcal_\psi^*\phi(x)\,ds\,dx\\
    &=\label{gliou-e28}\int_{0}^t\Ee \left(\int_{\rn}f(x) \Lcal_\psi^*\phi (x-X_s)\,dx\right) ds.
\end{align}
In the last equality we use that $(-X_t)_{t\geq 0}$ is the L\'evy process corresponding to $\Pcal_t^*$. The use of Fubini's theorem is justified by a calculation similar to the one in the following paragraph. Thus, if $\Lcal_\psi f=0$ weakly, we see that $\Pcal_tf=f$.

For the converse, we use Lemma~\ref{gliou-25} and the calculation
\begin{align*}
    \left|\int_{\rn} f(x) \Lcal_\psi^*\phi(x-y)\,dx\right|
    &\leq \int_{\rn}g((x-y)+y) \left|\Lcal_{\psi}^*\phi (x-y)\right| dx\\
    &\leq c g(y) \int_{\rn} g(x-y) \left|\Lcal_\psi^*\phi(x-y)\right| dx\\
    &= c g(y) \|g\Lcal_\psi^*\phi\|_{L^1}
\end{align*}
to see that  the function $y\mapsto \int_{\rn} f(x) \Lcal_\psi^*\phi(x-y)\,dx$ is continuous and bounded by $g$. Since $t\mapsto X_t$ is right-continuous and $(g(X_t))_{t\in [0,C]}$ is uniformly integrable, we can use Vitali's convergence theorem to conclude that the function
\begin{gather*}
    t\mapsto \Ee \left(\int_{\rn} f(x) \Lcal_\psi^*\phi(x-X_t)\,dx\right)
\end{gather*}
is right-continuous, and so
\begin{align*}
    \lim_{t\to 0}\frac{1}{t}\int_{0}^t \Ee \left(\int_{\rn} f(x) \Lcal_\psi^*\phi(x-X_s)\,dx\right) ds
    = \int_{\rn}f(x)\Lcal_\psi^*\phi(x)\,dx .
\end{align*}
Thus, we get from \eqref{gliou-e28} and $\Pcal_tf=f$ that $\Lcal_\psi f=0$ weakly.
\end{proof}

Up to now we have used the fact that the only real zero of the symbol $\psi$ is $\xi=0\in\rn$. We need to extend $\psi$ to strips in $\comp^n$\footnote{If a Levy exponent can be extended analytically to a strip in the complex domain, then the L\'{e}vy--Khintchine representation extends to these complex arguments, cf.\ \cite[Theorem 8.4.2]{lukacs70} for the one-dimensional case.}, and the following lemma shows when this is possible. Since $x\mapsto \eup^{\xi\cdot x}$ is a positive function, the integrals $\Ee \eup^{\xi\cdot X_t} = \int \eup^{\xi\cdot y}\,\mu_t(dy)$ are always defined in $[0,+\infty]$.
\begin{lemma}\label{gliou-31}
    Let $\Lcal_\psi$ be the generator of a L\'evy process $(X_t)_{t\geq 0}$ with characteristic exponent $\psi$ given by \eqref{vor-e06}.
    \begin{enumerate}
    \item\label{gliou-31-a} For every $t>0$ the set $F(t) = \{\xi\in\rn \mid \Ee \eup^{\xi\cdot X_t} < \infty\}$ coincides with $F = \{\xi\in\rn \mid \int_{|y|\geq 1} \eup^{\xi\cdot y}\,\nu(dy) < \infty\}$; moreover $F(t)$, hence $F$, is a non-empty convex set.

    \item\label{gliou-31-b} $\eta_0\in F$ if, and only if, $\psi$ can be analytically extended to the strip $\Sigma(\eta_0) = \{\xi - \iup\lambda\eta_0 \mid \zeta\in\rn, \lambda\in [0,1]\}$.

    \item\label{gliou-31-c}
        We have $\{\xi\in\rn \mid \Ee \eup^{\xi\cdot X_t} = 1\} = \{\eta\in\rn \mid \psi(-\iup\eta)=0\}$.
    \end{enumerate}
\end{lemma}
\begin{proof}
    The first part of the lemma follows immediately from the criterion \eqref{gliou-e23} on generalized moments with $g(y) = \eup^{\xi\cdot y}$. The convexity of $F(t)$ is a consequence of H\"{o}lder's inequality.

    For the second part, we replace in \eqref{vor-e06} $\xi$ by $\xi - \iup\eta$ to get (formally)
    \begin{align*}
        \psi(\xi - \iup\eta)
        &= - b\cdot \eta
        + \frac 12 Q\xi\cdot\xi - \frac 12 Q\eta\cdot\eta
        + \int\limits_{\rn\setminus\{0\}} \left(1-\cos(\xi\cdot x) \eup^{\eta\cdot x} + \eta\cdot x\I_{(0,1)}(|x|)\right)\nu(dx)\\
        &\qquad\mbox{}- \iup b\cdot \xi - \iup Q\xi\cdot\eta + \iup\int_{\rn\setminus\{0\}} \left(\xi\cdot x \I_{(0,1)}(|x|) - \sin(\xi\cdot x) \eup^{\eta\cdot x} \right)\nu(dx).
    \end{align*}
    From this expression it is immediate that $\psi(\xi- \iup\eta)\in\comp$ if, and only if, $\eta\in F$. Continuity, resp., one-sided continuity at the boundaries, is now a routine application of the dominated convergence theorem.

    Finally, for the last part we use that for a L\'evy process $\Ee \eup^{\iup\xi\cdot X_t} = \eup^{-t\psi(\xi)}$, hence $\Ee \eup^{\eta\cdot X_t} = \eup^{-t\psi(-\iup\eta)}$ whenever one side of the latter equality is well-defined.
\end{proof}

We are finally ready for the proof of the Liouville theorem for positive, exponentially bounded solutions of $\Lcal_\psi f = 0$.

\begin{theorem}\label{gliou-35}
    Let $\Lcal_\psi$ be the generator of a L\'evy process with characteristic exponent $\psi$ given by \eqref{vor-e06}. Assume that there exists a locally bounded, submultiplicative function $g:\rn\to [1,\infty)$ satisfying $\int_{|y|\geq 1} g(y)\nu(dy)<\infty$. The following assertions are equivalent:
    \begin{enumerate}
    \item\label{gliou-35-a}
    every measurable, positive and $g$-bounded function $0\leq f\leq g$ such that $\Lcal_\psi f = 0$ weakly is constant.

    \item\label{gliou-35-b}
    $\{\xi\in\rn \mid \psi(\xi) = 0\} = \{\eta\in\rn \mid \psi(-\iup\eta) = 0\} = \{0\}$.
    \end{enumerate}
\end{theorem}

\begin{remark}
    Our proof shows that the extension of $\psi$ to complex values in Theorem~\ref{gliou-35} is understood in the sense of Lemma~\ref{gliou-31}. That is, if $\psi$ cannot be extended to a non-degenerate strip in $\rn\times\iup\rn$, then the condition $\{\eta\in\rn \mid \psi(-\iup\eta) = 0\} = \{0\}$ trivially holds, and $\Lcal_\psi$ satisfies \ref{gliou-35-a}.
\end{remark}

\begin{proof}[Proof of Theorem~\ref{gliou-35}]
    As in Remark \ref{gen-03} we can assume that $f\in C^\infty(\rn)$. If the mollifier is compactly supported, we see that the smoothed-out $f$ is again bounded by the submultiplicative function $g$:
    \begin{gather*}
        f*j_\epsilon(x)
        = \int f(x-y)j_\epsilon(y)\,dy
        \leq \int g(x-y)j_\epsilon(y)\,dy
        \leq cg(x)\int g(-y)j_\epsilon(y)\,dy
        = c'g(x).
    \end{gather*}

    Assume first that \ref{gliou-35-a} holds. In particular, every positive, bounded function $f$ is $g$-bounded with $g\equiv\|f\|_\infty$. Thus, \ref{gliou-35-a} includes the Liouville property discussed in Theorem~\ref{prob-21}, and we conclude with Theorem~\ref{prob-21} that $\{\xi\in\rn \mid \psi(\xi)=0\}=\{0\}$. In particular, $\supp(\mu_t)$ generates for almost all  $t>0$ the whole group $(\rn,+)$, since our condition rules out all lattice laws.

    Using Lemma~\ref{gliou-29} we can reduce $\Lcal_\psi f = 0$ weakly to $\mathcal{P}_t f=f$ for all $t>0$. Thus, we can apply Theorem~\ref{gliou-22} to deduce that $f$ has the representation
    \begin{align*}
        f(x)=\int_{E(\mu_t)}\eup^{\xi\cdot x}\,\kappa(d\xi)
        \quad\text{with}\quad
        E(\mu_t) = \left\{\xi\in\rn \mid \Ee \eup^{\xi\cdot X_t} = \int_\rn \eup^{\xi\cdot x}\,\mu_t(dx)=1\right\}
    \end{align*}
    for some measure $\kappa$ with support in $E(\mu_t)$. From Lemma~\ref{gliou-31} we know that $E(\mu_t) = \{\eta\in\rn \mid \psi(-\iup\eta) = 0\}$; if $f$ is constant, the uniqueness of the Choquet representation shows that $E(\mu_t)=\{0\}$, and \ref{gliou-35-b} follows.

    Since $\{\xi\in\rn \mid \psi(\xi)=0\}$ is equivalent to $\supp(\mu_t)$ generating the group $(\rn,+)$, the above argument already shows the converse implication \ref{gliou-35-b}$\Rightarrow$\ref{gliou-35-a}.
\end{proof}

\begin{example}\ \ \textbf{a)}
    A typical example where the condition $\{0\}=\{\xi\in\rn \mid  \Ee \eup^{\xi\cdot X_t} = 1\}$ is violated is Brownian motion with drift. Consider $(B_t+\gamma t)_{t\in [0,\infty)}$ where $(B_t)_{t\geq 0}$ is a standard Brownian motion in $\rn$ and  $\gamma\neq 0$ is the drift. The generator is given by $\frac{1}{2}\Delta + \gamma\cdot \nabla$ and the symbol is $\psi(\xi) = \frac 12|\xi|^2 - \iup\gamma\cdot\xi$. Thus,
    \begin{align*}
        \left(\frac{1}{2}\Delta_x + \gamma\cdot \nabla_x\right)\eup^{\eta\cdot x}
        = \left(\frac{1}{2} |\eta|^2  + \gamma\cdot\eta \right)\eup^{\eta\cdot x}
        = \psi(-\iup\eta) \eup^{\eta\cdot x}.
    \end{align*}
    Choosing $\eta = -2\gamma$, we see that  there exists a  non-constant solution. Moreover, it is easy to see that  $\Ee \left(\eup^{\eta\cdot(B_t+\gamma t)}\right)\big|_{\eta=-2\gamma} = 1$ for all $t\in [0,\infty)$.
\medskip

\textbf{b)}
    If the symbol $\psi:\rn\to\real$ is real-valued, the L\'evy--Khintchine formula \eqref{vor-e06} becomes
    \begin{gather*}
        \psi(\xi) = \frac 12Q\xi\cdot\xi + \int_{\rn\setminus\{0\}} \left(1-\cos(\xi\cdot x)\right) \nu(dx)
    \intertext{and on the imaginary axis we have (at least formally)}
        \psi(-\iup\eta) = -\frac 12Q\eta\cdot\eta + \int_{\rn\setminus\{0\}} \left(1-\cosh(\eta\cdot x)\right) \nu(dx).
    \end{gather*}
    Since the integrand $1-\cosh u$ is smaller than or equal to $0$ and has exactly one zero at $u=0$, it is clear that $\{\eta\in\rn \mid \psi(-\iup\eta)=0\} = \{0\}$ always holds for such symbols.
\end{example}

We close this section  with a probabilistic argument which ensures that $\{\xi\in\rn \mid  \Ee \eup^{\xi\cdot X_t} = 1\}=\{0\}$. Recall that a L\'evy process is said to be \textbf{genuinely $n$-dimensional} if $X_t$ does not just take values in a hyperplane, i.e.\ if $\supp(\mu_t)$ generates the whole group $(\rn,+)$, cf.\ \cite[p.\ 156--7]{sato}.
\begin{corollary}\label{gliou-37}
    Let  $(X_t)_{t\geq 0}$ be a genuinely $n$-dimensional L\'evy process with characteristic exponent $\psi$ given by \eqref{vor-e06} such that $\{\xi\in\rn \mid \psi(\xi) = 0\}=\{0\}$. If the L\'evy process $(\beta\cdot X_t)_{t\geq 0}$ is recurrent for every $\beta\in\rn$, then the set $\{\xi\in\rn \mid  \Ee \eup^{\xi\cdot X_t} = 1\}$ is equal to $\{0\}$; equivalently, $\psi(-\iup\eta)$ does not have any other zero than $\eta=0$.
\end{corollary}
\begin{proof}
    Let $\beta \in\{\xi\in\rn \mid  \Ee \eup^{\xi\cdot X_t}=1\}\setminus\{0\}$. A short calculation shows that $(\exp(\beta \cdot X_t))_{t\geq 0}$ is a positive, right-continuous martingale satisfying $\Ee \exp( \beta\cdot X_t)=1$ for all $t\geq 0$. Since $(X_t)_{t\geq 0}$ is genuinely $n$-dimensional, the martingale is not a.s.\ constant.  Doob's martingale convergence theorem proves that $\exp(\beta\cdot X_t)$ converges a.s.\ to some finite random variable $Y_\infty$ as $t\to\infty$. On the other hand, since $(\beta\cdot X_t)_{t\geq 0}$ is a L\'evy process, $\exp(\beta \cdot X_t)$ can only converge to a finite limit if $\beta\cdot X_t\to -\infty$; this means that $(\beta\cdot X_t)_{t\geq 0}$ is transient.
\end{proof}

\section{Further notes and complements}\label{sec-notes}

Using \textbf{Bochner's subordination}, cf.\ \cite[Chapter 13]{SSV}, we can give a further characterization of the Liouville property. If $(\Pcal_t)_{t\geq 0}$ is any strongly continuous contraction semigroup and $(\gamma_t)_{t\geq 0}$ a vaguely continuous convolution semigroup of probability measures on the half-line $[0,\infty)$, then $\Tcal^h_t u := \int_0^\infty \Pcal_su\,\gamma_t(ds)$ (understood as a Bochner integral) is again a strongly continuous contraction semigroup. The superscript $h$ in $\Tcal^h_t$ denotes a \textbf{Bernstein function}. Bernstein functions uniquely characterize the convolution semigroup $(\gamma_t)_{t\geq 0}$ via the (one-sided) Laplace transform $\widetilde\gamma_t(\lambda) := \int_{[0,\infty)} \eup^{-\lambda s}\,\gamma_t(ds) = \eup^{-t h(\lambda)}$, and all Bernstein functions are of the form
\begin{gather}\label{notes-e40}
    h(\lambda) = a\lambda + \int_{(0,\infty)} \left(1-\eup^{-\lambda s}\right)\pi(ds)
\end{gather}
where $a\geq 0$ and $\pi$ is a measure on $(0,\infty)$ satisfying $\int_0^\infty \min\{s,1\}\,\pi(ds)<\infty$, see \cite[Chapter 3]{SSV}. Of the numerous properties of Bernstein functions let us only note that every $h\not\equiv 0$ is strictly monotone. Typical examples of Bernstein functions are fractional powers $h(\lambda)=\lambda^\alpha$, $0<\alpha<1$, ($a=0$, $\pi(ds) = \frac{\alpha}{\Gamma(1-\alpha)} s^{-1-\alpha}\,ds$), logarithms $h(\lambda)=\log(1+\lambda)$ ($a=0$, $\pi(ds) = s^{-1}\eup^{-s}\,ds$), `resolvents' $h(\lambda) = \frac{\lambda}{\tau+\lambda}$, $\tau>0$, ($a=0$, $\pi(ds) = \tau \eup^{-s\tau}\,ds$) and `semigroups' $h(\lambda) = 1-\eup^{-\lambda t}$, $t>0$, ($a=0$, $\pi(ds) = \delta_t(ds)$).

If we extend $h$ to the complex right-half plane $\HH = \{\lambda+\iup\eta \mid \lambda\geq 0, \eta\in\real\}$, we see from \eqref{notes-e40} that
\begin{gather}\label{notes-e42}
    \Re h(\zeta) = a\lambda + \int_{(0,\infty)} \left(1-\eup^{-\lambda s}\cos(\eta s)\right)\pi(ds),\quad \zeta=\lambda+\iup\eta \in\HH.
\end{gather}
Note that $\eup^{-\lambda s} < 1$ if $\lambda, s > 0$. Thus, if $a>0$, $h(\zeta)=0$ if, and only if, $\zeta=0$. If $a=0$ and $\lambda>0$, the inequality
\begin{gather*}
    1-\eup^{-\lambda s}\cos(\eta s) > 0,\quad \lambda, s>0
\end{gather*}
shows that $h(\zeta)=0$ can only happen if $\lambda=0$. In this case we also need that
\begin{gather*}
    \int_{(0,\infty)} \left(1-\cos(\eta s)\right)\pi(ds) = 0
\end{gather*}
which is only possible for $\eta\neq 0$ if $\supp\pi$ is discrete\footnote{If we compare \eqref{vor-e06} and \eqref{notes-e40} for $\lambda = \iup\eta$, we see that $\eta\mapsto h(\iup\eta)$ is a continuous and negative definite function, and the exact structure of $\supp\pi$ is given in Corollary~\ref{notes-45} below.}.

If $\psi$ is the symbol of a L\'evy process, then $h\circ\psi$ is also the symbol of a L\'evy process and $-h(-\Lcal_\psi) = \Lcal_{h\circ\psi}$ on $\Scal(\rn)$; here, $-h(-\Lcal_\psi)$ is understood as a function of the operator $-\Lcal_\psi$ in virtually any reasonable functional calculus sense, see \cite[Chapter 13]{SSV}. The following corollary can also be seen as a generalization of Lemma~\ref{prob-17}.

\begin{corollary}\label{notes-43}
    Let $\Lcal_\psi$ be the generator of a L\'evy process and $h$ a Bernstein function given by \eqref{notes-e40}. If $a\neq 0$ or if $\supp\pi$ is not discrete, then $\Lcal_\psi$ has the Liouville property \eqref{vor-e14} if, and only if, $\Lcal_{g\circ\psi}$ has the Liouville property.

    In particular, $\Lcal_\psi$ has the Liouville property if, and only if, for some \textup{(}or all\textup{)} $\tau>0$ the resolvent $\Rcal_\tau := (\tau\id-\Lcal_\psi)^{-1}$ enjoys the following property:
    \begin{gather}\label{notes-e44}
        f\in L^\infty(\rn)\;\;\text{and}\;\; \tau \Rcal_\tau f = f \implies f\equiv\textup{const}
    \end{gather}
    Moreover, $\Lcal_\psi$ has the Liouville property, if and only if, the semigroup $\Pcal_t$ enjoys the following property:
    \begin{gather}\label{notes-e46}
        f\in L^\infty(\rn)\;\;\text{and}\;\; \Pcal_t f = f \implies f\equiv\textup{const}
    \end{gather}
    for at least two values $t=t_1>0$ and $t=t_2>0$ such that $t_1/t_2\notin\rat$. If, in addition, $\Lcal_\psi$ is self-adjoint, i.e.\ if $\psi = \overline\psi$ is real-valued, then it is enough to have \eqref{notes-e46} for one $t>0$.
\end{corollary}
\begin{proof}
    The first assertion follows from the observation that $\{\zeta\in\HH \mid h(\zeta)=0\}=\{0\}$ implies $\{h\circ \psi = 0\}=\{\psi = 0\}$ and Theorem~\ref{prob-21}.

    For the second part use the Bernstein function $h(\lambda)=\frac{\lambda}{\tau+\lambda}$ and note that $\lambda=0$ is the only zero of $h$ in $\HH$; moreover $(\phi-\tau\Rcal_\tau\phi)\ \widehat{}= \psi (\tau+\psi)^{-1}\widehat\phi$.

    Note that $\widehat{\Pcal_t\phi} = \eup^{-t\psi}\widehat\phi$. Let $g_t = 1-\eup^{-t\lambda}$ and observe that $g_{t_1}(\zeta) = g_{t_2}(\zeta)=0$ for $\zeta\in\HH$ is only possible if $\zeta=0$---this is due to the assumption that $t_1/t_2\notin\rat$.

    Finally, if $\psi$ is real-valued, i.e.\ if $\Lcal_\psi$ is self-adjoint, there is no need to extend $h$ to $\HH$. In this case, $\{\lambda \geq 0 \mid h(\lambda)=0\}=\{0\}$ is always trivial  (because of the strict monotonicity of $h$)  and we can use the previous argument for $h(\lambda) = 1-\eup^{-\lambda t}$ and the semigroup $\Pcal_t$.
\end{proof}

The paper \cite{ali-et-al} by Alibaud \emph{et al.}\ contains another characterization of the Liouville property \eqref{vor-e14} for L\'evy generators using completely different methods: It is based on the characteristic triplet $(b,Q,\nu)$ appearing in the L\'evy--Khintchine representation \eqref{vor-e06} of $\psi$. If we combine our Theorem~\ref{prob-21} with the result of \cite{ali-et-al}, we arrive at an interesting description of the zero-set of a negative definite function.

We need the following notation from \cite{ali-et-al}.
Let $\Sigma$ be the positive semidefinite square root of $Q$ and denote by $\sigma_1,\dots,\sigma_n$ the column vectors of $\Sigma$.
Let $G_\nu = G(\supp\nu)$ be the smallest additive subgroup of $\rn$ containing $\supp\nu$, $V_\nu = \{x\in\overline{G_\nu} \mid tx\in\overline{G_\nu}\;\:\forall t\in\real\}$  ($\overline G$ stands for the closure of $G\subset\rn$),  $c_\nu = -\int_{\{|y|<1\}\setminus V_\nu} y\,\nu(dy)$ and $W_{\Sigma,b+c_\nu} = \mathrm{span}\{\sigma_1,\dots,\sigma_n, b+c_\nu\}$.
\begin{corollary}\label{notes-45}
    Let $\psi$ be the characteristic exponent of a L\'evy process given by \eqref{vor-e06}. Denote by $\{\psi = 0\}^{\boxperp}$ the orthogonal subgroup $\{x\in\rn \mid \forall \xi\in\{\psi=0\} : \eup^{\iup \xi\cdot x}=1\}$ of the additive group $\{\psi=0\}\subset\rn$. Then the following equality holds
    \begin{gather}
        \{\psi = 0\}^{\boxperp}
        = \overline{G_\nu + W_{\Sigma,b+c_\nu}}.
    \end{gather}
\end{corollary}
There is a further characterization of $\{\psi=0\}^{\boxperp}$ which can be found in Berg and Forst \cite[Proposition 8.27]{ber-for}. Denote by $\mu_t$ the probability measure such that $\widecheck \mu_t = \eup^{-t\psi}$, i.e.\ $(\mu_t)_{t\geq 0}$ is the family of transition probabilities of the L\'evy process with exponent $\psi$. Then $\{\psi=0\}^{\boxperp}$ is the smallest closed additive subgroup of $\rn$ which contains $\bigcup_{t>0} \supp(\mu_t)$.  This implies immediately the following result.
\begin{corollary}\label{notes-47}
    Assume that the L\'evy process with characteristic function $\psi$ has transition probabilities $(\mu_t)_{t\geq 0}$ such that for at least one $t_0>0$ the measure $\mu_{t_0}(dx) = p_{t_0}(x)\,dx$ has a strictly positive density, i.e.\ $p_{t_0}(x) > 0$ for all $x\in\rn$. Then the generator $\Lcal_\psi$ has the Liouville property \eqref{vor-e14}.
\end{corollary}

It is well known that powers of the Laplace operator also has the Liouville property. This carries over to generators of a L\'{e}vy process.
\begin{proposition}\label{notes-49}
    Let $\Lcal_\psi$ be the generator of a L\'evy process with characteristic exponent $\psi$ given by \eqref{vor-e06}. Then $\Lcal_\psi$ has the Liouville property \eqref{vor-e14} if, and only if, $\Lcal_\psi^n$ has the Liouville property for \textup{(}some or, equivalently,\textup{)} every $n\in\nat$.
\end{proposition}

\begin{proof}
    Assume that $\Lcal_\psi$ has the Liouville property. Let $n\geq 2$ and $f\in L^\infty(\rn)$ such that $\Lcal^n_\psi f=0$. As before, we may assume that $f\in C^\infty_b(\rn)$. We see that $\Lcal_\psi(\Lcal_\psi^{n-1}f)=0$, which implies that $\Lcal_\psi^{n-1}f=c$ for some constant $c\in\real$, hence we calculate that
    \begin{align*}
        \mathcal{P}_t \Lcal_\psi ^{n-2}f-\Lcal_\psi^{n-2}f=ct
    \end{align*}
    for all $t>0$, where we use Dynkin's formula. As the left-hand side is uniformly bounded in $t\in [0,\infty)$, we see that $c=0$. By using an induction argument, we conclude that $f$ is constant.

    Conversely, if $\Lcal_\psi$ does not have the Liouville property, there exits some non-constant function $f\in C^\infty_b(\rn)$ such that $\Lcal_\psi f=0$. Moreover, we see that $\Lcal_\psi f=0$ implies also $\Lcal^n_\psi f=0$.
\end{proof}

\noindent
    \emph{Note added in proof}: If $\psi$ is real-valued, i.e.\ if the corresponding L\'evy process $(X_t)_{t\geq 0}$ is symmetric, then the condition $\{\psi=0\}=\{0\}$ is also equivalent to the fact that the process has a unique invariant measure, cf.\ \cite{ying94}. We are grateful to Prof.\ P.J.\ Fitzsimmons for pointing this out to us.

\end{document}